\documentclass[]{article}

\usepackage{amsmath}
\usepackage{amssymb}
\usepackage{amsthm}

\usepackage{times}
\usepackage{geometry}
\usepackage[utf8]{inputenc}
\usepackage[T1]{fontenc}
\usepackage{tikz-cd}
\usepackage{hyperref}
\hypersetup{
	colorlinks,
	citecolor=blue,
	filecolor=black,
	linkcolor=black,
	urlcolor=magenta
}

\newtheorem{thm}{Theorem}[section]
\newtheorem{cor}[thm]{Corollary}
\newtheorem{lem}[thm]{Lemma}

\newtheorem*{thm*}{Theorem}

\newtheorem{prop}[thm]{Proposition}

\theoremstyle{definition}
\newtheorem{defin}[thm]{Definition}

\newtheorem*{claim}{Claim}

\newcommand{\forcing}{
	\operatorname{Fn}(\kappa,Metr,\omega)}

\newcommand{\rg}{\operatorname{rg}}
\newcommand{\dom}{\operatorname{dom}}
\newcommand{\metr}{(\mathcal{M},d)}
\newcommand{\rect}{(\mathcal{X},c)}
\newcommand{\K}{\mathcal{K}}
\newcommand{\ch}{\operatorname{CH}}

\newcommand{\ma}{\operatorname{MA}}
\newcommand{\zfc}{\operatorname{ZFC}}

\newcommand{\oca}{\operatorname{OCA}}
\title{What would the rational Urysohn space and the random graph look like if they were uncountable?}
\author{Ziemowit Kostana
	\footnote{Research of Z. Kostana was supported by the GAČR project EXPRO 20-31529X and RVO: 67985840, and by the European Research Council (grant
		agreement ERC-2018-StG 802756).}\\
	Institute of Mathematics
	Czech Academy of Sciences\\
	Žitná 25, 115 67 Prague, Czech Republic;\\
	University of Warsaw,\\
	Banacha 2, 02-097 Warsaw, Poland \\
	z.kostana@mimuw.edu.pl}

\begin{document}
	
	\maketitle
	
	\begin{abstract}
		Building on the work of Avraham, Rubin, and Shelah, we aim to build a variant of the Fra\"iss\'e theory for uncountable models built from finite submodels. With this aim, we generalize the notion of an increasing set of reals to other structures. As an application, we prove that the following is consistent: there exists an uncountable, separable metric space $X$ with rational distances, such that every uncountable partial 1-1 function from $X$ to $X$ is an isometry on an uncountable subset. We aim for a general theory of structures with this kind of properties. This includes results about the automorphism groups, and partial classification results.
	\end{abstract}
	
	
	\section{Introduction}
	\subsection{History and Motivation}
	
	It is well-known that $\ch$ entails the existence of $\omega_1$-saturated models of size $\omega_1$, for all first order theories admitting infinite models (\cite{hodges}, ch. 10). In many important cases, such models are unique up to isomorphism. This observation is a cornerstone of the Fra\"iss\'e-J\'onsson theory, developed in \cite{jonsson}. These models are built as $\omega_1$-chains of countable structures, and so it seems interesting to ask whether we can build uncountable models by "glueing together" finite structures. One indication that such theory should be possible to formulate is the celebrated theorem of Baumgartner, about separable $\omega_1$-dense linear orders. Recall that a linear order is $\omega_1$-dense if each non-empty open interval has cardinality $\omega_1$.

	\begin{thm*}[Baumgartner, \cite{baum}]
		It is consistent with $\zfc$ that there exists a unique up to isomorphism separable $\omega_1$-dense linear order.
	\end{thm*}
	
	Once the conclusion of this theorem holds, we can take any subfield of the reals of size $\omega_1$ as a model of this ordering. This shows that the (unique) $\omega_1$-dense separable linear order is homogeneous with respect to finite subsets (as is any subfield of the reals). Ideas of Baumgartner were extended by Avraham, Rubin, and Shelah in \cite{as} and \cite{ars}. Among other things, they show that Baumgartner's Theorem does not follow from $\ma+\lnot \ch$, but the latter axiom already ensures some approximation of the Baumgartner's Theorem. They introduce another axiom, $\oca_{ARS}$, and show that $\ma_{\omega_1}+\oca_{ARS}$ implies that either the conclusion of Baumgartner's Theorem holds, or there are up to isomorphism exactly three homogeneous separable $\omega_1$-dense linear orders (\cite{ars}, Sec. 6). This gives some insight that $\ma_{\omega_1}$ might be used in place of induction in some uncountable variant of the Fra\"iss\'e theory.
	
	\par We continue the line of research started in \cite{cohen-like}. We study the structures we introduced therein, and apply to them the ideas from \cite{as} and \cite{ars}. For example, the generic structures from \cite{cohen-like}, although initially rigid, become highly homogeneous in suitable forcing extensions satisfying $\ma_{\omega_1}$. Inductive arguments from the classical Fra\"iss\'e theory are replaced by Martin's Axiom, and in some cases it is enough to ensure some form of  uniqueness. 
	
	\par Even if aiming towards a new variant of the Fra\"iss\'e theory is perhaps too ambitious, we still find new applications for some techniques developed in \cite{as}. As an example, let us look at a Ramsey-type result from \cite{as}.
	
	\begin{thm*}[Thm. 2, \cite{as}]
		It is consistent with $\zfc+\ma+"2^{\omega}=\omega_2"$ that there exists an uncountable set $A\subseteq \mathbb{R}$ with the property that each uncountable  1-1 function $f\subseteq A\times A$ is strictly increasing on an uncountable set.
	\end{thm*}
	
	In this spirit, we prove in Section 2
	
	\begin{thm*} 
		It is consistent with $\zfc + \ma + "2^\omega=\omega_2"$ that there exists an uncountable, separable rational metric space $(X,d)$ such that each uncountable 1-1 function $f\subseteq X\times X$ is an isometry on an uncountable set.
	\end{thm*}
	
	This phenomenon is more general. In Section 3, we prove similar results for other classes of structures, although they are more technical, and require some introduction (what is a "separable" graph?). In Section 4, we prove some results on the number of isomorphism types of models we built.
	
	\subsection{Acknowledgements} The article is based on the results from the author's doctoral dissertation. I thank my advisor Wies\l aw Kubi\'s for his guidance, and creating a warm and welcoming atmosphere for the research work in the Institute of Mathematics of the Czech Academy of Sciences. Also, I thank Assaf Rinot for stimulating discussions that helped to improve the final presentation of the theory. I acknowledge that there is no conflict of interest.
	
	\section{Preliminaries}
	
	By an \emph{embedding} we always mean an isomorphism into its image. A subset of a model is a \emph{substructure} if the identity inclusion is an embedding. We will say that a language is \emph{purely relational} if it doesn't contain any function or constant symbols. We write $\dom{f}$ and $\rg{f}$ to denote respectively the domain and the image of a function $f$.
	
	\par We study structures added generically by forcings $$\operatorname{Fn}(S,\mathcal{K},\lambda),$$
	that we introduced in \cite{cohen-like}. The variable $\mathcal{K}$ always stands for a class of structures in some purely relational language. All unspecified forcing terminology follows \cite{kunen}.
	
	\begin{defin}
		Let $\lambda$ be an infinite cardinal, and $S$ an infinite set. For any model $A \in \mathcal{K}$ we denote by $F(A)$ its underlying set. We introduce the forcing $\operatorname{Fn}(S,\mathcal{K},\lambda)$ as the set
		$$\{ A \in \mathcal{K}|\; F(A) \in [S]^{<\lambda} \},$$
		ordered by the reversed substructure relation. 
	\end{defin}
	
	We proved in \cite{cohen-like} that models added generically by $$\operatorname{Fn}(\omega_1,\mathcal{K},\omega)$$ tend to be rigid, while those added by
	$$\operatorname{Fn}(\omega,\mathcal{K},\omega)$$
	are isomorphic to the usual Fra\"iss\'e limits of $\mathcal{K}$. We must impose some restrictions on the class $\mathcal{K}$ for $\operatorname{Fn}(\omega,\mathcal{K},\omega)$ to be an interesting forcing notion. The following is sufficient to ensure the c.c.c. property in all cases in the range of our interest. The name \emph{Splitting Property} was coined by W. Kubi\'s. Two embeddings $f:A\hookrightarrow B$ and $g:A\hookrightarrow C$ are \emph{isomorphic}, if there exists an isomorphism $h:B\hookrightarrow C$, such that $h\circ f = g$.
	
	\begin{defin}
		$\mathcal{K}$ has the Splitting Property (SP) if for all $R \in \mathcal{K}$, all pairs of isomorphic extensions $R\subseteq X$,  $R \subseteq Y$, there exists $Z\in \mathcal{K}$, containing $X \cup Y$.
	\end{defin}
	
	\begin{prop} \label{splitting}
		Let $\mathcal{K}$ be a class of structures in a purely relational language, that satisfies the following
		\begin{itemize}
			\item $\mathcal{K}$ has countably many isomorphism types of finite models,
			\item $\mathcal{K}$ is hereditary (if $A \in \mathcal{K},\; B\subseteq A,$ then $B \in \mathcal{K}$),
			\item $\mathcal{K}$ has the SP.
		\end{itemize}
		Then the forcing $\operatorname{Fn}(\lambda,\mathcal{K},\omega)$ has the c.c.c. for any cardinal $\lambda$.
	\end{prop}
	
	\begin{proof}
		Consider a family of conditions $\{{A}_\xi|\; \xi<\omega_1\}$. Using $\Delta$-system Lemma, we can trim the sequence, so that the sets $\{F({A}_\xi)|\; \xi<\omega_1\}$ form a $\Delta$-system with the root $K \in [\lambda]^{<\omega}$. There are at most countably many structures from $\K$ with the universe $K$, so we can assume that we have $\mathbb{K} \in \K$ with $F(\mathbb{K})=K$, and for all $\xi\neq \eta<\omega_1$
		$${A}_\xi \cap {A}_\eta = \mathbb{K}.$$
		
		We can also assume that the extensions $\mathbb{K}\subseteq {A}_\xi$ and $\mathbb{K}\subseteq {A}_\eta$ are isomorphic.
		Now we can use the SP for the diagram\\
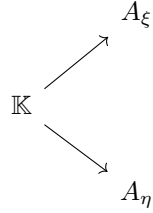
		\begin{center}
			\begin{tikzcd}
				&   {A}_\xi 
				&
				& \\
				\mathbb{K} \ar[ur] \ar[dr]
				&
				& 
				\\
				&    {A}_\eta
				&
				&
			\end{tikzcd}
		\end{center}
		
		to get a condition stronger than ${A}_\xi$ and ${A}_\eta$.
	\end{proof}

	\section{Following Avraham, Rubin, and Shelah}
	
	\subsection{Rectangular Metric Spaces}
	We adapt the technology from \cite{as} and \cite{ars} to prove that it is relatively consistent with $\zfc$ that there exists a separable rational metric space $(X,d)$ of size $\omega_1$, such that any uncountable 1-1 function from $X$ to itself is an isometry on an uncountable subset.
	The original result states that it is relatively consistent with $\zfc$ that there exists an uncountable set $A\subseteq \mathbb{R}$, such that every function $f:A\rightarrow A$ is non-decreasing on an uncountable set. Unlike the latter statement, our result does not seem to follow from Semiopen Coloring Axiom, Open Coloring Axiom or any other combinatorial principle described in \cite{ars}. 
	
	\par We first introduce a metric analog of a $k$-increasing linear order, introduced in \cite{as}.

	\begin{defin}
		Let $(X,d)$ be a metric space. 
		\begin{itemize}
			\item We call a pair of tuples $\overline{x}=(x_1,\ldots,x_n), \overline{y}=(y_1,\ldots,y_n) \in X^n$ \emph{alike} if they satisfy the following axioms:
			\begin{itemize}
				\item[A1]	$\forall \; i,j=1,\ldots,n \; (d(x_i,y_i)=d(x_j,y_j))$
				\item[A2]	$\forall \; i,j=1,\ldots,n \; (d(x_i,x_j)=d(y_i,y_j))$
				\item[A3]	$\forall \; i,j=1,\ldots,n \; (x_i\neq x_j \implies d(x_i,x_j)=d(x_i,y_j))$
			\end{itemize}
			We then write $\overline{x}\circledast \overline{y}$.
			\item We call $(X,d)$ \emph{rectangular} if it is uncountable, and for any sequence of pairwise disjoint tuples $\{(x_1^\xi,\ldots,x_n^\xi)|\; \xi<\omega_1\} \subseteq X^n$, there are $\xi \neq \eta <\omega_1$, such that $(x_1^\xi,\ldots, x_n^\xi) \circledast (x_1^\eta,\ldots,x_n^\eta)$.
		\end{itemize}
	\end{defin}
	
	Denote by $Metr$ the class of rational metric spaces. Keeping up with the general notation, $\forcing$ is the partial order
	$$\{(Y,d)|\; Y\in[\kappa]^{<\omega}, \text{ and $(Y,d)$ is a rational metric space} \},$$
	with the ordering relation being the reversed inclusion preserving the metric.
	
	\par We begin with a technical Lemma to ensure that we can amalgamate metric spaces in a specific way.
	
	\begin{lem} \label{metralike}
		
		Let $(R,d_R), (X,d_X), (Y,d_Y)$ be finite metric spaces, such that $(X,d_X)\cap(Y,d_Y)=(R,d_R)$, and suppose that $h:(X,d_X)\rightarrow (Y,d_Y)$ is an isometric bijection, which is identity on $R$. \\
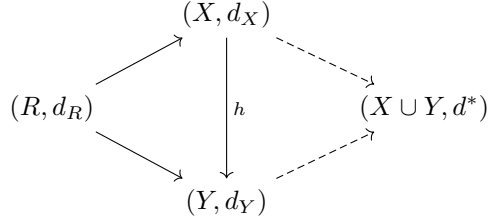
		\begin{center}
			\begin{tikzcd}
				& (X,d_X) \ar[dr,dashed] \ar[dd, "h"]
				&
				& \\
				(R,d_R) \ar[ur] \ar[dr]
				&
				& (X\cup Y,d^*)
				\\
				& (Y,d_Y) \ar[ur,dashed]
				&
				&
			\end{tikzcd}
		\end{center}
		
		Then there exists a metric $d^*$ on $X\cup Y$ extending both $d_X$ and $d_Y$, such that if $(x_1,\ldots,x_n)$ is a bijective enumeration of $X\setminus R$, then $(x_1,\ldots,x_n)\circledast (h(x_1),\ldots,h(x_n))$.

	\end{lem}
	
	\begin{proof}
		Let $s=\min\{ d_X(x,x')|\; x\neq x' \in X \}$. Given that $d^*$ must extend the metrics of $X$ and $Y$, we must set the distances between elements	from $X\setminus R$ and $Y\setminus R$. Therefore we set $d^*(x,h(x))=s$ and $d^*(x,h(x'))=d_X(x,x')$, for all $x\neq x' \in X \setminus R$. A standard computation shows that this definition gives a well-defined metric structure on $X\cup Y$, satisfying the required conditions.
	\end{proof}
	
	In the light of Proposition \ref{splitting}, we have the c.c.c. property for forcings $\forcing$. Even more generally, for any countable set $K\subseteq [0,\infty)$ the proof of Lemma \ref{metralike} shows that the class of finite metric spaces with distances in $K$ has the SP.
	
	\begin{prop} \label{propx1}
		For any uncountable cardinal $\kappa$, $\forcing \Vdash \text{"$(\kappa,\dot{d})$ is rectangular"}$.
	\end{prop}
	
	\begin{proof}
		
		Let $\{(\dot{x}_1^\xi,\ldots, \dot{x}_n^\xi)|\;\xi<\omega_1 \}$ be a sequence of $\forcing$-names for pairwise disjoint $n$-tuples from $(\kappa,\dot{d})$. Fix a condition $p$. For every $\xi <\omega_1$, we find a condition $p_\xi=(p_\xi,d_\xi) \le p$, deciding values of $\dot{x}^\xi_i$ and $\dot{d} \restriction \{ x^\xi_1,
		\ldots, x^\xi_n \} \times \{ x^\xi_1,
		\ldots, x^\xi_n \}$. We choose an uncountable set $S \subseteq \kappa$, satisfying the following conditions:
		\begin{itemize}
			\item $\{p_\xi|\; \xi \in S\}$ is a $\Delta$-system with the root $R=(R,d_R)$,
			\item $\forall \; \xi, \eta \in S$ there exists an isometry $h:p_\xi \rightarrow p_\eta$, which is identity on $R$, and $h(x^\xi_i)=x^\eta_i$, for $i=1,\ldots,n$, as shown in the diagram.
			
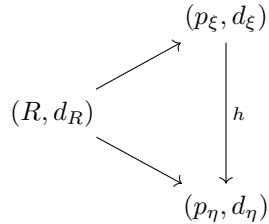
\begin{center}
				\begin{tikzcd}
					& (p_\xi,d_\xi)  \ar[dd, "h"]
					&
					& \\
					(R,d_R) \ar[ur] \ar[dr]
					&
					& 
					\\
					& (p_\eta,d_\eta) 
					&
					&
				\end{tikzcd}
			\end{center}
		\end{itemize}
		
		This can be easily done, since given any point $\alpha \in \omega_1$ outside of $R$, there are only countably many possible configurations of distances between this point and $R$. Now we choose $\xi \neq \eta \in S$, and apply Lemma \ref{metralike} to obtain $q\le p_\xi,p_\eta$, which forces that $(x^\xi_1,\ldots, x^\xi_n) \circledast (x^\eta_1,\ldots,x^\eta_n)$. 
	\end{proof}

	Let $\metr$ be the metric space we added to our model by $\forcing$. Our next task is to force $\ma_{\omega_1}$, while preserving $\metr$ being rectangular. Following the ideas from \cite{as}, we introduce a special class of partial orders.
	
	\begin{defin} \label{defappropriate}
		A partial order $\mathbb{P}$ satisfies $\metr$\emph{-c.c.} if given any natural number $n>0$, for each family consisting of pairwise disjoint tuples 
		$$\{(p_\xi,x_1^\xi,\ldots,x_n^\xi)|\; \xi < \omega_1 \} \subseteq \mathbb{P} \times \metr^n,$$
		there exist $\xi\neq \eta <\omega_1$, such that $p_\xi$ and $p_\eta$ are comparable, and
		$$(x_1^\xi,\ldots,x_n^\xi) \circledast (x_1^\eta,\ldots,x_n^\eta).$$
	\end{defin}
	
	\begin{prop} \label{propx4}
		If $\mathbb{P}$ satisfies $\metr${-c.c.} then $\mathbb{P} \Vdash \text{"$\metr$ is rectangular"}$.	
	\end{prop}
	
	\begin{proof}
		Fix a sequence of $\mathbb{P}$-names $\{ (\dot{x}_1^\xi,\ldots,\dot{x}_n^\xi)|\; \xi <\omega_1 \}\subseteq \metr^n$ for pairwise disjoint $n$-tuples. For a given condition $p\in \mathbb{P}$, and $\xi<\omega_1$, we fix a condition $p_\xi \le p$ deciding $(\dot{x}_1^\xi,\ldots,\dot{x}_n^\xi)$. Then we apply $\metr${-c.c.} to the family $\{(p_\xi,x_1^\xi,\ldots,x_n^\xi)|\; \xi <\omega_1 \}$. This way we obtain $\xi \neq \eta < \omega_1$, and $q \le p_\xi,p_\eta$, such that $q \Vdash 
		(\dot{x}_1^\xi,\ldots,\dot{x}_n^\xi) \circledast (\dot{x}_1^\eta,\ldots,\dot{x}_n^\eta)$.
	\end{proof}
	
	The argument showing that $\metr${-c.c.} is preserved under iterations is really not different than the one showing that the c.c.c. is preserved, applied for example in \cite{kunen}. We include it for completeness.
	
	\begin{prop} \label{propx5}
		Finite support iterations of $\metr${-c.c.} posets are $\metr${-c.c.}.
	\end{prop}
	
	\begin{proof}
		
		Assume that $\mathbb{P}$ is $\metr${-c.c.}, and $\mathbb{P}\Vdash \text{"$\dot{\mathbb{Q}}$ is $\metr${-c.c.}"}$. Take a sequence 
		$$\{ (p_\xi,\dot{q}_\xi,x_1^\xi,\ldots,x_n^\xi)|\; \xi<\omega_1 \},$$
		and towards contradiction assume that it witnesses $\mathbb{P}\ast \dot{\mathbb{Q}}$ not being $\metr${-c.c.}. Let $\dot{\sigma}$ be a $\mathbb{P}$-name defined $\dot{\sigma}=\{(\xi,p_\xi)|\; \xi <\omega_1\}$. If $G\subseteq \mathbb{P}$ is an $M$-generic filter, then
		$$M[G]\models (\xi \in \sigma \iff p_\xi \in G).$$
		We claim that in $M[G]$, for any two conditions $\eta,\xi \in \sigma$, if $(x_1^\xi,\ldots,x_n^\xi)\circledast (x_1^\eta,\ldots,x_n^\eta)$, then $q_\xi$ and $q_\eta$ are inconsistent. For otherwise, there exist $q\le q_\eta,q_\xi$, and $p \in G$, which forces it. Since $p_\xi$ and $p_\eta$ are in $G$, which is a filter, we may choose $p\le p_\xi,p_\eta$. Then $(p,\dot{q})\le (p_\xi,\dot{q}_\xi),(p_\eta,\dot{q}_\eta)$, and $(x_1^\xi,\ldots,x_n^\xi)\circledast (x_1^\eta,\ldots,x_n^\eta)$ contrary to the choice of the sequence $\{(p_\xi,\dot{q}_\xi,x_1^\xi,\ldots,x_n^\xi)|\; \xi<\omega_1 \}$.
		
		\par $G$ was an arbitrary generic filter, and remember that $\mathbb{P}\Vdash \text{"$\dot{\mathbb{Q}}$ is $\metr${-c.c.}"}$. The conclusion of this is that $\mathbb{P} \Vdash |\dot{\sigma}|< \omega_1$. Then there exists a $\mathbb{P}$-name for a countable ordinal $\dot{a}$, for which
		$\mathbb{P}\Vdash \dot{\sigma}\subseteq \dot{a}$. Since $\mathbb{P}$ is c.c.c. there are only countably many possible values of $\dot{a}$, and by taking supremum of them, we can replace $\dot{a}$ by a canonical name $a$. But note that $p_{a}\Vdash a \in \dot{\sigma}$. This is a contradiction, and it finishes the proof for $\mathbb{P}\ast \dot{\mathbb{Q}}$.
		
		\par Consider now a finite support iteration of an infinite length $\theta$ of $\metr${-c.c.} forcings $\overline{\mathbb{P}}=\{\mathbb{P}_\alpha\ast \dot{\mathbb{Q}}_\alpha|\; \alpha < \theta\}$. We prove by induction on $\theta$, that $\overline{\mathbb{P}}$ is $\metr${-c.c.}. The successor step has just been taken care of, so suppose that the conclusion holds for any ordinal less than $\theta$, and $\theta$ is limit. Take any disjoint sequence $\{(\overline{p}_\alpha,\overline{x}_\alpha) |\; \alpha<\omega_1\} \subseteq \overline{\mathbb{P}}\times \metr^n$. We can assume that the supports of conditions $\overline{p}_\alpha$ form a $\Delta$-system with the root $R$, and that for any $\alpha$, $\operatorname{supp}{\overline{p}_\alpha}\setminus R$ is above $R$. Let $\delta = \max{R}$. There exist two different $\alpha,\beta < \omega_1$, for which the relation $\overline{x}_\alpha\circledast \overline{x}_\beta$ holds and $\overline{p}_\alpha \restriction \delta$ is comparable with $\overline{p}_\beta\restriction \delta$. From the general theory of finite support iterations it follows that $\overline{p}_\alpha$ and $\overline{p}_\beta$ are comparable.
	\end{proof}
	
	It may look suspicious that in the proof above we never actually used $\mathbb{P}$ being $\metr${-c.c.}, only c.c.c. However one can verify that if some c.c.c. forcing forces a poset to be $\metr${-c.c.}, then it must be $\metr${-c.c.} itself. So, we could have just as well assume that $\mathbb{P}$ is c.c.c. The immediate consequence of this proposition is
	
	\begin{lem}
		It is consistent with $\zfc + \ma(\text{$\metr${-c.c.}}) + ``2^{\omega}=\omega_2"$ that $\metr$ is rectangular.
	\end{lem}
	
	In fact, full Martin's Axiom will hold in such model, what we are going to show. Recall that if $\mathbb{P}$ is a forcing notion, a set $D\subseteq \mathbb{P}$ is \emph{predense} if each element of $\mathbb{P}$ is comparable with an element of $D$. A set $D$ is \emph{predense below $p$} if each element of $\mathbb{P}$ stronger that $p$ is comparable with an element of $D$.
	
	\begin{lem} \label{predense}
		Let $\mathbb{P}$ be a c.c.c. forcing notion of size $\omega_1$. There exists $p \in \mathbb{P}$ and a family of $\omega_1$ many subsets of $\mathbb{P}$, predense below $p$, such that any filter $G \subseteq \mathbb{P}$ containing $p$ and intersecting all of them is uncountable.
	\end{lem}
	
	\begin{proof}
		Enumerate bijectively $\mathbb{P}$ as $\{p_\gamma|\; \gamma<\omega_1\}$. Let $D_\alpha= \{p_\gamma|\; \alpha<\gamma<\omega_1\}$. We aim to find $p \in \mathbb{P}$ such that uncountably many of sets $D_\alpha$ are predense below $p$ -- clearly this will finish the proof. If $p$ with this property doesn't exist, then the following assertion holds:
		
		$$\forall{\gamma<\omega_1}\; \exists{\gamma'>\gamma}\;\text{"$D_{\gamma'}$ is not predense below $p_\gamma$"},$$
		and consequently
		
		$$\forall{\gamma<\omega_1}\; \exists{p_\gamma'\le p_\gamma}\; \exists{\gamma'>\gamma}\; \forall{\eta>\gamma'}\; p_\gamma'\bot p_\eta.$$
		
		Using this, we can easily define an uncountable antichain $\{p_\gamma|\; \gamma \in E\}$, ensuring at each step of induction that
		$$\forall{\gamma \in E}\; \exists{\gamma'>\gamma}\; \forall{\eta>\gamma'}\; p_\gamma\bot p_\eta.$$
	\end{proof}
	
	\begin{lem} \label{metrknaster}
		$\ma_{\omega_1}(\text{$\metr${-c.c.}})$ implies that any family of pairwise disjoint tuples\\ $\{(x_1^\xi,\ldots,x_n^\xi)|\; \xi<\omega_1 \}\subseteq \metr^n$ contains an uncountable subfamily of pairwise alike tuples.
	\end{lem}
	
	\begin{proof}
		Let $\mathbb{P}=\{ F \in [\omega_1]^{<\omega}|\; \forall{\xi\neq \eta \in F}\; \overline{x}_\xi \circledast \overline{x}_\eta  \}$, where $\overline{x}_\xi=(x_1^\xi,\ldots,x_n^\xi)$. The ordering is given by the reversed inclusion. We claim that $\mathbb{P}$ is $\metr${-c.c.}. Fix an uncountable family $\{F_\alpha|\; \alpha<\omega_1 \} \subseteq \mathbb{P}$, and a family of tuples $\{(v^\alpha_1,\ldots,v^\alpha_s)|\; \alpha<\omega_1\} \subseteq \metr^s$. Without loss of generality we may assume that $$F_\alpha=\{e_1,\ldots,e_k,e^\alpha_{k+1},\ldots,e^\alpha_{k+m} \},$$ where sets $\{e^\alpha_{k+1},\ldots,e^\alpha_{k+m} \}$ are pairwise disjoint for different $\alpha$. For each $\alpha$, let $\overline{y}_\alpha \in \metr^{n\cdot m+s}$ be a concatenation of all tuples $\overline{x}_{e^\alpha_{k+1}},\ldots,\overline{x}_{e^\alpha_{k+m}}$, and $(v^\alpha_1,\ldots,v^\alpha_s)$. There exists $\alpha \neq \beta <\omega_1$, such that $\overline{y}_\alpha \circledast \overline{y}_\beta$, and we claim that they witness the fact that $\mathbb{P}$ is $\metr${-c.c.}. Let us write:
		
		$$\overline{y}_\alpha=(x_1^1,\ldots, x^1_n,x^2_1,\ldots,x^2_n,\ldots,x^m_1,\ldots,x^m_n,v_1,\,\ldots,v_s),$$
		$$\overline{y}_\beta=(y_1^1,\ldots, y^1_n,y^2_1,\ldots,y^2_n,\ldots,y^m_1,\ldots,y^m_n,u_1,\,\ldots,u_s).$$
		
		What is clear, is that $(v_1,\ldots,v_s) \circledast (u_1,\ldots,u_s)$, and for all $k\le m$, $(x_1^k,\ldots,x_n^k)\circledast (y_1^k,\ldots, y_n^k)$. What is not clear, is that in this case the $\circledast$ relation is "shift-invariant", i.e. for all $1\le p\neq r \le m$ we have $(x_1^p,\ldots,x_n^p)\circledast (y_1^r,\ldots, y_n^r)$. We will check that this is the case.
		
		\begin{enumerate}
			
			\item [A1] For all $1\le r\neq p \le m$, and $1\le i,j\le n$, we have
			\begin{align*} 
				d(x_i^p,y_i^r)=d(x_i^p,x_i^r) & \text{  by A3 for $\overline{y}_\alpha$ and $\overline{y}_\beta$,} \\
				d(x_i^p,x_i^r)=d(x_j^p,x_j^r) & \text{  by A1 for $(x_1^p,\ldots,x_n^p)$ and $(x_1^r,\ldots,x_n^r)$,} \\
				d(x_j^p,x_j^r)=d(x_j^p,y_j^r) & \text{  by A3 for $\overline{y}_\alpha$ and $\overline{y}_\beta$}.
			\end{align*}
			
			\item [A2] For all $1\le r\neq p \le m$, and $1\le i,j\le n$, we have
			\begin{align*}
				d(x_i^p,x_j^p)=d(x_i^r,x_j^r) & \text{  by A2 for $(x_1^p,\ldots,x_n^p)$, $(x_1^r,\ldots,x_n^r)$,} \\
				d(x_i^r,x_j^r)=d(y_i^r,y_j^r) & \text{  by A2 for $\overline{y}_\alpha$, $\overline{y}_\beta$.} \\
			\end{align*}

			\item[A3] If $x^r_i \neq x^r_j$ then 
			\begin{align*}
				d(x_i^r,x_j^r)=d(x_i^r,x_j^p) & \text{  by A3 for $(x_1^p,\ldots,x_n^p)$, $(x_1^r,\ldots,x_n^r)$,} \\
				d(x_i^r,x_j^p)=d(x_i^r,y_j^p) & \text{  by A3 for $\overline{y}_\alpha$, $\overline{y}_\beta$.} \\
			\end{align*}
			
		\end{enumerate}
		
		Therefore $F_\alpha\cup F_\beta \in \mathbb{P}$, and this concludes the proof that $\mathbb{P}$ is $\metr${-c.c.}. Notice that all singletons belong to $\mathbb{P}$, so $|\mathbb{P}| = \omega_1$. Applying Martin's Axiom to the family of predense sets given by the Lemma \ref{predense}, we find an uncountable filter $G\subseteq \mathbb{P}$. The set $\{\overline{x}_\alpha|\; \alpha \in \bigcup G \}$ is an uncountable family of tuples, and each two of them are alike.
	\end{proof}
	
	\begin{prop}
		$\ma_{\omega_1}(\text{$\metr${-c.c.}})$ implies that any c.c.c. partial order of size $\omega_1$ satisfies $\metr${-c.c.}.
	\end{prop}
	
	\begin{proof}
		Suppose that $\mathbb{P}$ is a c.c.c. partial order of cardinality $\omega_1$, and fix some disjoint family $\{ (p_\xi,x_1^\xi,\ldots,x_n^\xi)| \;  \xi < \omega_1 \}\subseteq \mathbb{P} \times \metr^n$. If $\ma_{\omega_1}(\text{$\metr${-c.c.}})$ holds, we can assume that all tuples $(p_\xi,x_1^\xi,\ldots,x_n^\xi)$ are pairwise alike, and since $\mathbb{P}$ is c.c.c. we will find two $p_\xi$ and $p_\eta$, which are comparable.
	\end{proof}
	
	The immediate consequence is
	
	\begin{thm}  \label{mainthm0}
		It is consistent with $\zfc+ \ma + ``2^{\omega}=\omega_2"$ that $\metr$ is rectangular.
	\end{thm}
	

	\begin{thm}  \label{mainthm1}
		Let $(X,d)$ be a rectangular rational metric space of size $\omega_1$. $\ma_{\omega_1}$ implies that any uncountable 1-1 function $f \subseteq X \times X$ is an isometry on an uncountable set.
	\end{thm}
	
	\begin{proof}
		
		Let $f \subseteq X \times X$ be an uncountable 1-1 function. Consider the partial order
		$$\mathbb{P}_f=(\{E \in [\dom{f}]^{<\omega}|\; f\restriction E \text{ is an isometry} \},\subseteq).$$
		
		We will check that $\mathbb{P}_f$ is c.c.c. Take a sequence $\{e_\xi|\; \xi <\omega_1 \} \subseteq \mathbb{P}_f$. Applying $\Delta$-system Lemma we may assume that
		\begin{itemize}
			\item $\forall \; \xi <\omega_1 \; |e_\xi|=m$,
			\item $\forall \; \xi \neq \eta <\omega_1 \;  e_\xi\cap e_\eta = r$,
			\item $e_\xi=(e_1^\xi,\ldots,e_l^\xi,e_{l+1}^\xi,\ldots,e_m^\xi)$, where $\{e_1^\xi,\ldots,e_l^\xi\}=r$. 
		\end{itemize}

		Look at the family of tuples 
		$$\{ (e_{l+1}^\xi,\ldots,e_m^\xi,f(e_{l+1}^\xi),\ldots,f(e_m^\xi))|\; \xi <\omega_1 \}.$$
		Using for example $\Delta$-system Lemma one can easily trim this sequence so that all tuples are pairwise disjoint. Now, given that $(X,d)$ is rectangular, we find $\xi \neq \eta < \omega_1$, such that
		
		$$(e_{l+1}^\xi,\ldots,e_m^\xi,f(e_{l+1}^\xi),\ldots,f(e_m^\xi))\circledast
		(e_{l+1}^\eta,\ldots,e_m^\eta,f(e_{l+1}^\eta),\ldots,f(e_m^\eta)).$$ 
		
		We must check that $e_\xi$ and $e_\eta$ are comparable, that is $f \restriction (e_\xi \cup e_\eta)$ is an isometry. But notice that for $i\neq j=l+1,\ldots,m$,
		
		$$d(e_i^\xi,e_j^\eta)=d(e_i^\xi,e_j^\xi)=d(f(e_i^\xi),f(e_j^\xi))
		=d(f(e_i^\xi),f(e_j^\eta)).$$
		This proves that $e_\xi \cup e_\eta \in \mathbb{P}_f$. 
		\par Notice that all singletons belong to $\mathbb{P}_f$, so $|\mathbb{P}| = \omega_1$. Applying Martin's Axiom to the family of predense sets given by the Lemma \ref{predense}, we find an uncountable filter $G\subseteq \mathbb{P}_f$. The set $\bigcup\{E|\; E \in G \}$ is an uncountable set on which $f$ is an isometry.
	\end{proof}
	
	\begin{cor} \label{corhom}
		It is consistent with $\zfc+\ma + ``2^\omega=\omega_2"$ that there exists an uncountable, separable (even hereditarily separably saturated) rational metric space $(X,d)$ such that each uncountable 1-1 function $f\subseteq X\times X$ is an isometry on an uncountable set.
	\end{cor}
	
	Under $\ch$, the above assertion fails.
	
	\begin{thm} Assume $\ch$, and let $(X,d)$ be any uncountable, separable metric space. There is a 1-1 function $h:X\rightarrow X$ that is not isometry on any uncountable subset of $(X,d)$.
	\end{thm}
	\begin{proof}
		Since $2^\omega=\omega_1$, we know that $X$ has size $\omega_1$, and so we can assume that $X=\omega_1$. Let us fix an enumeration $\{f_\alpha|\; \alpha<\omega_1\}$ of all isometries between countable subsets of $(\omega_1,d)$. We denote by $\overline{f}_\beta$ the unique continuous extension of $f_\beta$ to the closure of $\dom(f_\beta)$. We define $h:\omega_1 \rightarrow \omega_1$ by induction, ensuring at step $\alpha$ that
		\begin{itemize}
			\item $h(\alpha)\neq h(\beta)$, for every $\beta<\alpha$,
			\item $h(\alpha) \neq \overline{f}_\beta(\alpha)$, for every $\beta<\alpha$, for which $\alpha \in \overline{\dom(f_\beta)}$.
		\end{itemize}
		Towards a contradiction, suppose now that $A\subseteq \omega_1$ is uncoutnable, and $h\restriction A$ is an isometry. There exists $\beta<\omega$, for which $h\restriction A $ extends $f_\beta$, and $\dom(f_\beta)$ is dense in $A$. It follows that
		$$f_\beta \subseteq h\restriction A \subseteq \overline{f}_\beta.$$
		Now, pick any $\alpha \in A \setminus (\beta+1)$. We have
		$$h(\alpha)\neq \overline{f}_\beta(\alpha)=h(\alpha).$$
		This shows that $h\restriction A$ cannot be an isometry.
	\end{proof}

	\subsection{Rectangular Models for General Binary Relations}
	
	The reader might have noticed that in the previous section the triangle inequality for the space $\metr$ was never applied. The crucial property we used was a variant of the SP, which ensures that "remainders" will be alike (see Lemma \ref{metralike}). In this section we will define a version of the SP that allows to proceed with the proof of Theorem \ref{mainthm1} in case of other classes of structures. One assumption which seems hard to be removed is that the language consists only of binary relational symbols.
	
	\par Each language consisting of finitely many binary relational symbols can be identified with a finite coloring of ordered pairs -- given a model $A$, we assign to each element of $A^2$ its isomorphism type. There are only finitely many symbols in the language, so this coloring is indeed finite, and moreover, it determines the model $A$ completely. Also, any function is a homomorphism precisely when it preserves this coloring. This observation allows to generalize results from Section 2 to other classes, besides metric spaces. In fact, the finiteness of language is not relevant as long as there are only countably many isomorphism types of 2-element models. Let $\mathcal{K}$ be some class of structures in a countable language $\{\mathcal{R}_i\}_{i<\omega}$ consisting of binary relational symbols. Assume also that $\mathcal{K}$ has only countably many isomorphism types of finite models. Let $c$ be the corresponding coloring of pairs in models from $\mathcal{K}$ -- by the remark above we may forget about the relational symbols, and think of models from $\mathcal{K}$ as having only one "relation", namely $c$. We introduce the $\circledast$ relation by axioms similar to the metric case, but unlike metrics, the coloring $c$ might not be symmetric. This is the only significant difference.

	\begin{defin}
		Let $X \in \mathcal{K}$, and $(x_1,\ldots,x_n),(y_1,\ldots,y_n) \in X^n$ be disjoint. We will say that they are \emph{alike}, and write $(x_1,\ldots,x_n) \circledast (y_1,\ldots,y_n)$, if the following axioms are satisfied
		\begin{enumerate}
			\item[A1a]	$\forall \; i,j=1,\ldots,n \; c(x_i,y_i)=c(x_j,y_j)$
			\item[A1b]	$\forall \; i,j=1,\ldots,n \; c(y_i,x_i)=c(y_j,x_j)$
			\item[A2a]	$\forall \; i,j=1,\ldots,n \; c(x_i,x_j)=c(y_i,y_j)$
			\item[A3a]	$\forall \; i,j=1,\ldots,n \; (x_i\neq x_j \implies c(x_i,x_j)=c(x_i,y_j)=c(y_i,x_j))$
		\end{enumerate}
	\end{defin}
	
	If all relations $\mathcal{R}_k$ are anti-reflexive ($\forall x\; \lnot \mathcal{R}_k(x,x)$), then we can omit the clause $(x_i \neq x_j) \implies$ in A3a. It is standard to check that 
	$$(x_1,\ldots,x_n) \circledast (y_1,\ldots,y_n) \iff (y_1,\ldots,y_n) \circledast (x_1,\ldots,x_n).$$
	
	\begin{defin}
		$X \in \mathcal{K}$ is \emph{rectangular} if $|X|>\omega$, and for any family of pairwise disjoint tuples $\{(x_1^\xi,\ldots,x_n^\xi)|\; \xi<\omega_1\} \subseteq X^n$, there exist $\xi\neq \eta<\omega_1$, such that $(x_1^\xi,\ldots,x_n^\xi) \circledast (x_1^\eta,\ldots,x_n^\eta).$
	\end{defin}
	
	\begin{defin}
		$\mathcal{K}$ has the \emph{Rectangular Splitting Property} (RSP) if for all $R \in \mathcal{K}$, for all pairs of isomorphic extensions $R\subseteq X$,  $R \subseteq Y$, with the corresponding isomorphism $h:X \rightarrow Y$, there exists $Z\in \mathcal{K}$, with the universe $X \cup Y$, such that for any sequence $(x_1,\ldots,x_n)$ enumerating bijectively $X \setminus R$, $Z$ satisfies $(x_1,\ldots,x_n)\circledast (h(x_1),\ldots,h(x_n))$.
	\end{defin}
	
	The RSP ensures the conclusion of Lemma \ref{metralike}. Moreover, since the definition of $\circledast$ relation is independent of the ordering of the tuples, if there exists an enumeration $(x_1,\ldots,x_n)$ like above, it can be replaced by any other enumeration, even not 1-1. 
	
	\begin{thm} \label{mainthm3}
		If $\mathcal{K}$ has the RSP then $\operatorname{Fn}(\kappa,\mathcal{K},\omega)$ forces the generic structure to be rectangular, for any uncountable cardinal $\kappa$.
	\end{thm}
	
	\begin{proof}
		Exactly like the proof of Proposition \ref{propx1}.
	\end{proof}

	Let us fix a rectangular model $\rect$. The $\rect${-c.c.} is defined in the same way as in Definition \ref{defappropriate}:
	
	\begin{defin}
		A partial order $\mathbb{P}$ satisfies \emph{$\rect$-c.c.} if given any natural number $n>0$, for each disjoint family 
		$\{(p_\xi,x_1^\xi,\ldots,x_n^\xi)|\; \xi < \omega_1 \} \subseteq \mathbb{P} \times \rect^n$, there exist $\xi\neq \eta <\omega_1$, such that $p_\xi$ and $p_\eta$ are comparable, and
		$(x_1^\xi,\ldots,x_n^\xi) \circledast (x_1^\eta,\ldots,x_n^\eta)$.
	\end{defin}
	
	\begin{prop}
		If $\mathbb{P}$ satisfies $\rect$-c.c. then $\mathbb{P} \Vdash \text{ "$\rect$ is rectangular".}$
	\end{prop}
	
	\begin{proof}
		Exactly like the proof of Proposition \ref{propx4}.
	\end{proof}
	
	\begin{prop} \label{generalknaster}
		$\ma_{\omega_1}(\text{$\rect$-c.c.})$ implies that any family of pairwise disjoint tuples\\ $\{(x_1^\xi,\ldots,x_n^\xi)|\; \xi<\omega_1 \}\subseteq \rect^n$ contains an uncountable subfamily of pairwise alike tuples.
	\end{prop}
	\begin{proof}
		Same as Lemma \ref{metrknaster}.
	\end{proof}
	
	What follows, is a generalization of Theorem \ref{mainthm0}.
	
	\begin{thm} \label{mainthm2}
		It is consistent with $\zfc+\ma + ``2^{\omega}=\omega_2"$ that $\rect$ is rectangular.
	\end{thm}

	The proof that the $\rect$-c.c. is preserved under finite support iterations is also the same. 
	
	\begin{prop} \label{propx6}
		Let $(X,c) \in \mathcal{K}$ be a rectangular model of size $\omega_1$. $\ma_{\omega_1}$ implies that any uncountable 1-1 function $f \subseteq X \times X$ is a homomorphism on an uncountable set.
	\end{prop}
	
	\begin{proof}
		Almost exactly like the proof of Theorem \ref{mainthm1} -- the only difference is that we write $c$ instead of $d$, \emph{isomorphism} instead of \emph{isometry}, and in the end use A1a and A1b instead of A1, since $c$ might not be symmetric.
	\end{proof}
	
	\begin{prop}
		The classes of graphs, directed graphs, tournaments, linear orders, and partial orders have the RSP.
	\end{prop}
	
	\begin{proof}
		We will prove the RSP for linear and partial orders. Arguments for other classes are easy and left to the reader. For each partial order $\le$ there exists a corresponding quasi-ordering relation $<$, which is anti-reflexive. 
		
		\par Suppose we have a diagram of linear quasi-orders $R,X,Y$, and $h:X\rightarrow Y$, like in the definition of the RSP. We define a quasi-ordering on $Z=X\cup Y$, extending both $<_X$ and $<_Y$, by conditions:
		\begin{itemize}
			\item $\forall{i<\omega}\; (x_i < h(x_i))$,
			\item $\forall{i\neq j<\omega}\; (x_i<h(x_j) \iff x_i <_X x_j).$
		\end{itemize}
		Clearly $<$ is an anti-reflexive relation on $Z$. For checking transitivity we must go through several (somewhat boring) cases.
		
		\begin{enumerate}
			\item $x_i<h(x_j),\; h(x_j)<x_k$. Either $x_i=x_j$ or $x_i<x_j$. In the first case $h(x_i)<x_k$, and so $x_i<x_k$. In the second, $x_i<x_k$, and $x_i<x_k$ follows from transitivity of $<$ on $X$.
			
			\item $h(x_i)<x_j,\; x_j<x_k$. In this case $x_i<x_j$, so $x_i<x_k$, and $h(x_i)<x_k$.
			
			\item $x_i<x_j,\; x_j<h(x_k)$. Either $x_j=x_k$ or $x_j<x_k$. In both cases $x_i<x_k$, so $x_i<h(x_k)$. 
			
			\item $h(x_i)<h(x_j),\; h(x_j)<x_k$. In this case $x_i<x_j$, and $x_j<x_k$. By transitivity $x_i<x_k$, and $h(x_i)<x_k$ follows.
			
			\item $x_i<h(x_j),\; h(x_j)<h(x_k)$. If $x_i<x_j$, then we proceed like before. If $x_i=x_j$, $h(x_i)<h(x_k)$. It follows that $x_i<x_k$, and $x_i<h(x_k)$.
			
			\item $h(x_i)<x_j,\; x_j<h(x_k)$. We see that $x_i<x_j$. If $x_j<x_k$, we use transitivity of $<$ on $X$. If $x_j=x_k$, then $h(x_i)<x_k$, and so $x_i<x_k$. It follows that $h(x_i)<h(x_k)$. 
			
		\end{enumerate}
		
		The proof for partial orders is strictly simpler -- we define the quasi-ordering by conditions:
		
		\begin{itemize}
			\item $\forall{i<\omega}\quad x_i \text{ is incomparable with } h(x_i)$,
			\item $\forall{i,j<\omega}\quad x_i<h(x_j) \iff x_i < x_j,$
			\item $\forall{i,j<\omega}\quad \lnot h(x_i)<x_j.$
		\end{itemize}
		
		Verification of transitivity is a run through the same cases, except this time we do not have to care if $x_i,x_j,x_k$ are distinct.
	\end{proof}
	
	For each class $\mathcal{K}$ having the RSP, one can prove a variant of Corollary \ref{corhom}. We could of course state a general theorem, after introducing a notion of "separable model" for arbitrary binary relational class. We will refrain from doing so, and provide just two such variants as an illustration. Interested reader will easily formulate corresponding results for tournaments, directed graphs, etc.
	\newpage
	
	\begin{thm} Each of the following is consistent with $\zfc+\ma+``2^\omega=\omega_2"$:
		\begin{enumerate}
			\item There exists a graph $G$ of size $\omega_1$, with a countable subset $D\subseteq G$, satisfying the following properties:
			\begin{itemize}
				\item For all pairs of disjoint finite subsets $A,B \subseteq G$, there exists a vertex $d \in D$, connected with each point in $A$, and with no point in $B$.
				\item Each uncountable 1-1 function $f\subseteq G\times G$ is a graph homomorphism on an uncountable set.	
			\end{itemize}
			\item (Avraham-Shelah, \cite{as}) There exists a separable, $\omega_1$-dense linear order\\ $(L,\le)$, such that each uncountable 1-1 function $f \subseteq L\times L$ is order preserving on an uncountable subset.
		\end{enumerate}
	\end{thm}
	
	The first point of this theorem is also a consequence of Theorem \ref{mainthm1} -- we can turn a metric space into a graph by connecting two vertices iff the distance between them is $\ge 1$. 
	
	\subsection{HSS Models}
	
	After forcing with $\operatorname{Fn}(\omega_1,\mathcal{LO},\omega)$, each infinite subset of $\omega_1$ from the ground model becomes a dense subset of the generic structure. This can be checked by a standard density argument, and analogous property holds for other classes. Specifically, uncountable generic models contain many countable subsets which realize every finite type. We present one way to capture this property. It will be useful in proving classification theorems in the next Section. 
	
	\begin{defin} \label{saturateddefinition1}
		Let $X$ be a structure in some relational language. A subset $D \subseteq X$ is a \emph{saturating subset} if for any finite subset $E\subseteq X$, for any single-point extension $E\subseteq E\cup \{f \}$, there exists $d \in D$ such that the following diagram commutes
		\begin{center}
			\begin{tikzcd}
				& E\cup\{f\}  \ar[dd, "f\mapsto d"]
				&
				& \\
				E \ar[ur] \ar[dr]
				&
				& 
				\\
				& E\cup\{d\} 
				&
				&
			\end{tikzcd}
		\end{center}
	\end{defin}
	
	To phrase it shortly, every finite type can be realized in $D$. In case of metric spaces, it this is just to say that any finite configuration of distances from points in $X$ can be realized inside $D$. A saturating subset of a metric space is dense in a very strong sense.
	
	\begin{defin} \label{saturateddefinition2}
		Let $X$ be any structure in a relational language.
		\begin{enumerate}
			\item $X$ is \emph{separably saturated} if it has a countable saturating subset.
			\item $X$ is \emph{hereditarily separably saturated} (HSS) if for any countable subset $A \subseteq X$, $X \setminus A$ is separably saturated.
		\end{enumerate}
	\end{defin}
	
	\begin{prop} \label{HSS}
		If $\mathcal{K}$ is a class with the RSP then $\operatorname{Fn}(\kappa,\mathcal{K},\omega)$ forces the generic structure to be HSS, for any uncountable cardinal $\kappa$.
	\end{prop}
	\begin{proof}
		A standard density argument shows that every infinite set from the ground model is a saturating subset.
	\end{proof}

	\section{Classification Results}
	
	If the uncountable models we are considering are to resemble Fra\"iss\'e limits, we should be able to do some kind of "back-and-forth" arguments, like in the classical theory. There is no way inductive arguments can work, but thanks to Theorem \ref{mainthm2} we can rely on Martin's Axiom instead of induction. 
	
	\subsection{HSS Rectangular Metric Spaces}
	
	We will be looking at metric spaces with distances in a given countable set $K\subseteq [0,\infty)$.
	
	\begin{thm} \label{mimimalmetr}
		Assume $\ma_{\omega_1}$. Let $(X,d)$ be any rectangular HSS metric space of size $\omega_1$, with distances in $K$. Let $Y \subseteq X$ be any HSS uncountable subspace. Then $X$ and $Y$ are isometric.
	\end{thm}
	
	\begin{proof}
		Since $X$ is hereditarily separably saturated, we can decompose it into a disjoint union of countable saturating subsets $\{X_\alpha\}_{\alpha<\omega_1}$. Of course we can do the same with $Y$, so let us write $Y=\displaystyle{\bigcup_{\alpha<\omega_1} Y_\alpha}$. Let $\mathbb{P}$ consist of finite partial isometries between $X$ and $Y$, which map elements from $X_\alpha$ to $Y_\alpha$ for all $\alpha<\omega_1$. It is standard to check that the following sets are dense for $x\in X$, $y\in Y$:
		$$D_x=\{p \in \mathbb{P}|\; x \in \dom{p} \},$$
		$$E_y=\{p \in \mathbb{P}|\; y \in \rg{p} \}.$$
		
		We are left with the task of verifying the c.c.c. property. Fix any uncountable subset $\{p_\gamma|\;\gamma<\omega_1 \}\subseteq \mathbb{P}$. Using $\Delta$-system Lemma we can write
		$$\dom{p_\gamma}=(x_1,\ldots,x_k,x_{k+1}^\gamma,\ldots,x_{m}^\gamma),$$
		$$\rg{p_\gamma}=(y_1,\ldots,y_k,y_{k+1}^\gamma,\ldots,y_{m}^\gamma),$$
		where tuples $(x_{k+1}^\gamma,\ldots,x_{m}^\gamma)$ are pairwise disjoint, and moreover $p_\gamma(x_i)=y_i$, and $p_\gamma(x_i^\gamma)=y_i^\gamma$ for each $\gamma<\omega_1$. Recall that $X$ is rectangular, so we can find $\xi\neq \eta <\omega_1$, such that
		$$(x^\xi_{k+1},\ldots,x_{m}^\xi,y_{k+1}^\xi,\ldots,y_{m}^\xi)\circledast (x^\eta_{k+1},\ldots,x_{m}^\eta,y_{k+1}^\eta,\ldots,y_{m}^\eta).$$
		If $k< i\neq j \le m$, then
		\begin{align*}
			d(x_i^\eta,x_j^\xi)= & d(x_i^\eta,x_j^\eta)= \\
			d(y_i^\eta,y_j^\eta)= & d(y_i^\eta,y_j^\xi)
		\end{align*}
		Also for $k<i\le m$
		\begin{align*}
			d(x_i^\eta,x_i^\xi)= & d(y_i^\eta,y_i^\xi)
		\end{align*}
		Clearly $p_\xi \cup p_\eta \in \mathbb{P}$.
	\end{proof}
	
	\begin{cor}
		Assume $\ma_{\omega_1}$. Let $(X,d)$ be any rectangular HSS metric space of size $\omega_1$, with distances in $K$. $X$ can be decomposed into a disjoint union of $\lambda$ many its isometric copies for any $\lambda \in \{2,\ldots,\omega_1\}$.
	\end{cor}
	
	\begin{proof}
		Let $X=\displaystyle{\bigcup_{\alpha<\omega_1}X _\alpha }$ be a decomposition of $X$ into countable saturating subsets. Let $\omega_1=\displaystyle{\bigcup_{\alpha<\lambda}A_\alpha}$ be a decomposition of $\omega_1$ into pairwise disjoint uncountable subsets. For any $\gamma<\lambda$ the space $$X_\gamma=\displaystyle{\bigcup_{\alpha\in A_\gamma}X_\alpha }$$
		is isometric to $(X,d)$.
	\end{proof}
	
	After taking $X=Y$ and obvious adjustments to the forcing used, we obtain the classical homogeneity.
	
	\begin{cor}
		Assume $\ma_{\omega_1}$. If $(X,d)$ is any rectangular HSS metric space of size $\omega_1$, with distances in $K$, then any finite partial isometry of $(X,d)$ extends to a full isometry.
	\end{cor}
	
	Still, this kind of homogeneity is not sufficient to prove uniqueness, like for the rational Urysohn space.
	
	\begin{thm}
		Assume $\ma_{\omega_1}$. If there exists a rectangular HSS rational metric space of size $\omega_1$, then there exist infinitely many pairwise non-isometric such spaces.
	\end{thm}
	
	\begin{proof}
		
		If $(X,d)$ is a rectangular, hereditarily separably saturated rational metric space of size $\omega_1$, we introduce a family of metrics on $X$
		$$d_k(x,y)=k\cdot d(x,y),$$
		for positive integers $k$. If $k<l$ are positive integers, the spaces $(X,d_k)$ and $(X,d_l)$ are not isometric. Indeed, if $f:X\hookrightarrow X$ is a bijection, then by Theorem \ref{mainthm1}, $f$ is an isometry of the space $(X,d)$ on some pair of distinct points $x,y \in X$. Then $d_k(x,y)=k\cdot d(x,y)=k\cdot d(f(x),f(y))<l\cdot d(f(x),f(y))=d_l(f(x),f(y))$. 
	\end{proof}
	
	In contrast to Theorem \ref{mimimalmetr} we have
	\begin{thm}
		Assume $\ch$, and suppose $(X,d)$ is a HSS metric space of size $\omega_1$, with distances in $K$. There is an uncountable HSS subspace $Y\subseteq X$ not isometric with $X$.
	\end{thm}
	\begin{proof}
		Let $\{f_\alpha|\; \alpha<\omega_1\}$ be an enumeration of all isometries between countable subspaces of $(X,d)$. We inductively choose saturating sets $Y_\alpha \in [X]^\omega$, so that for every $\gamma<\omega_1$
		
		$$Y_\gamma \cap (\bigcup_{\beta<\gamma}Y_\beta \cup \bigcup_{\beta \le \gamma}\overline{f}_\beta[\gamma+1])=\emptyset,$$
		where $\overline{f}_\beta$ is the unique continuous extension of $f_\beta$ to the closure of its domain.
		
		The choice of each $Y_\gamma$ is a straightforward application of the HSS property. Let us put $Y=\displaystyle{\bigcup_{\gamma<\omega_1}Y_\gamma}$, and suppose towards a contradiction that $h:X\rightarrow Y$ is an isometry. Since $(X,d)$ is separable, there is $\beta<\omega_1$ such that $h$ extends $f_\beta$, and $\dom(f_\beta)$ is dense in $X$. Therefore we have $$f_\beta \subseteq h \subseteq \overline{f}_\beta.$$
		
		Now, since $h$ is 1-1, it follows from Fodor's Lemma that for some $\alpha>\beta$ we have $h(\alpha) \in Y_\gamma$, where $\gamma \ge \alpha$. It follows that $h(\alpha)=\overline{f}_\beta(\alpha) \in \overline{f}_\beta[\gamma+1]$, and so
		$$h(\alpha)\in Y_\gamma \cap \overline{f}_\beta[\gamma+1],$$
		which contradicts the choice of $Y_\gamma$. Since $Y$ is a disjoint union of $\omega_1$ many saturating subsets of $X$, it is straightforward that $Y$ itself is HSS.
	\end{proof}
	
	\subsection{HSS Rectangular Graphs}
	
	If we take $K=\{0,1,2\}$, metric spaces with distances in $K$ are graphs -- think of two points in distance $1$ as connected, and two points in distance $2$ as not connected. Notions of a saturating set and separably saturated space translate to the following.
	
	\begin{defin} \label{saturateddefinition3}
		Let $G$ be a graph. A subset $D \subseteq G$ is a \emph{saturating subset} of $G$ if for any pair of disjoint finite subsets $A,B \subseteq G$, there exists $d \in D\setminus (A\cup B)$ which is connected with each vertex in $A$ and with no vertex in $B$.
	\end{defin}
	
	\begin{defin} \label{saturateddefinition4}
		If $G$ is any graph, then
		\begin{enumerate}
			\item $G$ is \emph{separably saturated} if it has a countable saturating subset.
			\item $G$ is \emph{hereditarily separably saturated} (HSS) if for any countable subset $E \subseteq G$, the graph $G \setminus E$ is separably saturated.
		\end{enumerate}
	\end{defin}
	
	\begin{prop}
		$\operatorname{Fn}(\omega_1,Graphs,\omega) \Vdash \text{"$(\omega_1,\dot{E})$ is a rectangular HSS graph"}$.
	\end{prop}
	\begin{proof}
		Rectangularity is a consequence of Theorem \ref{mainthm3}, HSS follows from Proposition \ref{HSS}.
	\end{proof}
	
	From Theorem \ref{mimimalmetr}, with $K=\{0,1,2\}$, it follows that HSS rectangular graphs are in certain sense minimal.
	
	\begin{thm}
		Assume $\ma_{\omega_1}$. If $G$ is a HSS rectangular graph of size $\omega_1$, then each uncountable HSS subgraph of $G$ is isomorphic with $G$.
	\end{thm}

	\begin{thm} \label{graphuniqueness}
		Assume $\ma_{\omega_1}$. If $G$ and $H$ are two HSS rectangular graphs of size $\omega_1$, then $G\simeq H$ if and only if $G$ and $H^c$ do not contain a common uncountable subgraph ($H^c$ denotes the complement of $H$).
	\end{thm}
	
	\begin{proof} $\text{  }$\\
		
		"$\Rightarrow$." We must show that $G$ and $G^c$ do not contain a common uncountable subgraph. Fix arbitrary graph $F$ of size $\omega_1$, and suppose towards contradiction that there exists a pair of embeddings $i:F\hookrightarrow G$, $j:F\hookrightarrow G^c$. The partial function given by $i(\alpha) \mapsto j(\alpha)$ is a bijection between uncountable subsets of $G$. By Proposition \ref{propx6} it is a homomorphism on some pair of points $\alpha \neq \beta$. This contradicts the choice of $i$ and $j$.\\ 
		
		"$\Leftarrow$." Assume that there is no uncountable graph which embeds both into $G$ and $H$. We proceed like in the proof of Theorem \ref{mimimalmetr}. $G$ and $H$ are HSS, so we can decompose them into disjoint unions of countable saturating subsets $\{G_\alpha\}_{\alpha<\omega_1}$, and $\{H_\alpha\}_{\alpha<\omega_1}$ respectively. Let $\mathbb{P}$ consist of finite partial isomorphisms between $G$ and $H$, which map elements from $G_\alpha$ to $H_\alpha$ for all $\alpha<\omega_1$. The following sets are dense for $g\in G$, $h\in H$:
		$$D_g=\{p \in \mathbb{P}|\; g \in \dom{p} \},$$
		$$E_h=\{p \in \mathbb{P}|\; h \in \rg{p} \}.$$
		
		We verify the c.c.c. property. Fix any uncountable subset $\{p_\gamma|\;\gamma<\omega_1 \}\subseteq \mathbb{P}$. Using $\Delta$-system Lemma we can write
		$$\dom{p_\gamma}=(x_1,\ldots,x_k,x_{k+1}^\gamma,\ldots,x_{m}^\gamma),$$
		$$\rg{p_\gamma}=(y_1,\ldots,y_k,y_{k+1}^\gamma,\ldots,y_{m}^\gamma),$$
		
		where tuples $(x_{k+1}^\gamma,\ldots,x_{m}^\gamma)$ are pairwise disjoint, and moreover $p_\gamma(x_i)=y_i$, and $p_\gamma(x_i^\gamma)=y_i^\gamma$ for each $\gamma<\omega_1$. By rectangularity and Proposition \ref{generalknaster} we can assume that tuples $$\{(x_{k+1}^\gamma,\ldots,x_{m}^\gamma)|\; \gamma<\omega_1\},$$ as well as $$\{(y_{k+1}^\gamma,\ldots,y_{m}^\gamma)|\; \gamma<\omega_1\}$$ are pairwise alike.

		For all $k< i\neq j \le m$, and $\xi\neq \eta < \omega_1$ we have
		\begin{align*}
			x_i^\eta \; E\; x_j^\xi \iff & x_i^\eta \; E \; x_j^\eta \iff \\
			y_i^\eta \; E \; y_j^\eta \iff & y_i^\eta \; E \; y_j^\xi
		\end{align*}
		What remains is the case $i=j$. But look at the function given by $x_{k+1}^\eta \mapsto y_{k+1}^\eta$, for $\eta < \omega_1$. This is a bijection between an uncountable subgraph of $G$ and an uncountable subgraph of $H$, or equivalently, $H^c$. By our assumption, it cannot be a homomorphism into $H^c$. We conclude that there are $\eta \neq \xi <\omega_1$ such that
		$$x_{k+1}^\eta \; E \; x_{k+1}^\xi \iff y_{k+1}^\eta \; E \; y_{k+1}^\xi,$$
		and by rectangularity
		$$x_{i}^\eta \; E \; x_{i}^\xi \iff y_{i}^\eta \; E \; y_{i}^\xi,$$
		for all $i=k+1,\ldots,m$. Clearly $p_\xi \cup p_\eta \in \mathbb{P}$.
	\end{proof}
	
	\begin{cor}
		Assume $\ma_{\omega_1}$. If there exists a HSS rectangular graph of size $\omega_1$ then it is unique up to taking the complement.
	\end{cor}
	
	\begin{proof}
		Fix two HSS rectangular graphs $G$ and $H$, both of size $\omega_1$. It is sufficient to prove that either no uncountable graph $F$ can be embedded both into $G$ and $H$, or no uncountable graph $F$ can be embedded both into $G$ and $H^c$. Suppose towards contradiction that $F_0,F_1\subseteq G$ are uncountable subgraphs, and there exist embeddings $i_0:F_0 \hookrightarrow H$, $i_1:F_1 \hookrightarrow H^c$. The function given by $i_0(f)\mapsto i_1(f)$ is a bijection between uncountable subsets of $H$. By Proposition \ref{propx6}, on some two points it must be a homomorphism. But this contradicts the choice of $i_0$ and $i_1$.
	\end{proof}
	
	\begin{thm} \label{cliquethm}
		Assume $\ma_{\omega_1}$. If $G$ is a rectangular graph of size $\omega_1$, then either $G$ contains an uncountable clique or uncountable anticlique.
	\end{thm}
	\begin{proof}
		In the light of Proposition \ref{propx6}, $G$ clearly can't contain both. Assume that $G$ doesn't contain an uncountable anticlique. We can represent $G$ as $(\omega_1,E)$, and consider the partial order 
		$$\mathbb{P}=\{F\subseteq \omega_1|\; F\text{ is a finite clique in }G \},$$
		ordered by reversed inclusion. If we can show that $\mathbb{P}$ is c.c.c, Lemma \ref{predense} will provide us with an uncountable clique in $G$. Suppose that $\{F_\xi|\; \xi<\omega_1 \}$ is an uncountable subset of $\mathbb{P}$. We can write
		$$F_\xi=(f_1,\ldots,f_k,f_{k+1}^\xi,\ldots,f_m^\xi),$$
		where tuples $(f_{k+1}^\xi,\ldots,f_m^\xi)$ are pairwise disjoint. By the virtue of Martin's Axiom, and Proposition \ref{generalknaster}, we can also assume that they are pairwise alike. By our assumption the set $\{f_{k+1}^\xi|\;\xi<\omega_1 \}$ is not an anticlique, so we will find $\xi\neq \eta <\omega_1$, such that $f_{k+1}^\eta \; E \; f_{k+1}^\xi$. It is now standard to check that $F_\eta \cup F_\xi \in \mathbb{P}$.
	\end{proof}

	\subsection{Separable $\omega_1$-dense Linear Orders}
	
	What do rectangular linear orders look like? After unwinding the axioms for the $\circledast$ relation, we see that
	$(x_1,\ldots,x_n)\circledast (y_1,\ldots,y_n)$ translates to the following three axioms:
	
	\begin{itemize}
		\item[A1c]	$\forall \; i,j=1,\ldots,n \; (x_i<y_i\iff x_j<y_j)$
		\item[A2c]	$\forall \; i,j=1,\ldots,n \; (x_i<x_j \iff y_i<y_j)$
		\item[A3c]	$\forall \; i,j=1,\ldots,n \; ((x_i\neq x_j)\implies (x_i<x_j\iff x_i<y_j\iff y_i<x_j))$
	\end{itemize}
	
	If we omitted A3c, we would obtain what the authors of \cite{ars} call an \emph{increasing order} (actually, one can show that in the class of separable, dense linear orders these two notions coincide). An order added by $\operatorname{Fn}(\omega_1,\mathcal{LO},\omega)$ is a rectangular, separable, $\omega_1$-dense linear order, which under $\ma_{\omega_1}$ is also homogeneous. $\ma_{\omega_1}$ imposes a great deal of regularity on the class of separable, homogeneous $\omega_1$-dense linear orders -- for example two such orderings are isomorphic precisely when they are bi-embeddable. Implications of $\ma_{\omega_1}$ for this class, as well as other axioms like $\oca_{ARS}$, have been extensively studied in \cite{ars}. Not surprisingly, many properties of rectangular linear orders resemble those of graphs.
	
	\begin{thm}\label{ihaveenough}
		Assume $\ma_{\omega_1}$, and suppose that $L$ is a separable, $\omega_1$-dense, rectangular linear order. If $K$ is any other separable, $\omega_1$-dense, rectangular linear order, then $K\simeq L$ if and only if $K$ and $L^*$ do not contain a common uncountable suborder ($L^*$ denotes $L$ with the reversed ordering).
	\end{thm}

	\begin{proof} $\text{  }$\\
		
		"$\Rightarrow$." We must show that there is no uncountable strictly decreasing function $f\subseteq L \times L$. But this follows directly from Proposition \ref{propx6}. \\
		
		"$\Leftarrow$." Assume that there is no uncountable linear order which embeds both into $L$ and $L^*$. We proceed like in the proofs of Theorems \ref{mimimalmetr} and \ref{graphuniqueness}. We can decompose $L$ and $K$ into disjoint unions of countable dense subsets $\{L_\alpha\}_{\alpha<\omega_1}$, and $\{K_\alpha\}_{\alpha<\omega_1}$ respectively. Let $\mathbb{P}$ consist of finite partial isomorphisms between $L$ and $K$ which map elements from $K_\alpha$ to $L_\alpha$, for all $\alpha<\omega_1$. The following sets are dense for $l\in L$, $k\in K$:
		$$D_k=\{p \in \mathbb{P}|\; k \in \dom{p} \},$$
		$$E_l=\{p \in \mathbb{P}|\; l \in \rg{p} \}.$$
		
		We verify the c.c.c. property. Fix any uncountable subset $\{p_\gamma|\;\gamma<\omega_1 \}\subseteq \mathbb{P}$. Using $\Delta$-system Lemma, we can write
		$$\dom{p_\gamma}=(x_1,\ldots,x_k,x_{k+1}^\gamma,\ldots,x_{m}^\gamma),$$
		$$\rg{p_\gamma}=(y_1,\ldots,y_k,y_{k+1}^\gamma,\ldots,y_{m}^\gamma),$$	
		where tuples $(x_{k+1}^\gamma,\ldots,x_{m}^\gamma)$ are pairwise disjoint, and moreover $p_\gamma(x_i)=y_i$, and $p_\gamma(x_i^\gamma)=y_i^\gamma$, for each $\gamma<\omega_1$. By rectangularity and Proposition \ref{generalknaster} we can assume, that tuples $$\{(x_{k+1}^\gamma,\ldots,x_{m}^\gamma)|\; \gamma<\omega_1\},$$ as well as $$\{(y_{k+1}^\gamma,\ldots,y_{m}^\gamma)|\; \gamma<\omega_1\},$$ are pairwise alike.

		For all $k< i\neq j \le m$, and $\xi\neq \eta < \omega_1$ we have
		\begin{align*}
			x_i^\eta < x_j^\xi \iff & x_i^\eta < x_j^\eta \iff \\
			y_i^\eta < y_j^\eta \iff & y_i^\eta < y_j^\xi
		\end{align*}
		What remains is the case $i=j$. Look at the function given by $x_{k+1}^\eta \mapsto y_{k+1}^\eta$, for $\eta < \omega_1$. This is a bijection between an uncountable subset of $L$ and an uncountable subset of $K$. By our assumption it cannot be decreasing, so there are $\eta \neq \xi <\omega_1$ such that
		$$x_{k+1}^\eta < x_{k+1}^\xi \iff y_{k+1}^\eta < y_{k+1}^\xi,$$
		and by rectangularity
		$$x_{i}^\eta \; < \; x_{i}^\xi \iff y_{i}^\eta \; < \; y_{i}^\xi,$$
		for all $i=k+1,\ldots,m$. Clearly $p_\xi \cup p_\eta \in \mathbb{P}$.	
	\end{proof}
	
	Just like in the case of graphs, one can easily prove
	
	\begin{cor}
		Assume $\ma_{\omega_1}$. If there exists a rectangular, separable, $\omega_1$-dense linear order then it is unique up to reversing the order.
	\end{cor}
	
	These orders are also mimimal in the same sense as metric spaces from Theorem \ref{mimimalmetr}. The reader will have no difficulty in adjusting its proof to obtain
	
	\begin{thm}
		Assume $\ma_{\omega_1}$ and let $K$, $L$ be a pair of separable, $\omega_1$-dense linear orders. If $K$ embeds into $L$, and $L$ is rectangular, then $L$ and $K$ are isomorphic.
	\end{thm}
	
	\section{More About Rectangularity}

	We will be considering models of the form $\chi=(\omega_1,c)$, where $c:\omega_1 \times \omega_1 \rightarrow K$ for some, usually finite, set $K$. A set $A \subseteq \omega_1$ is called \emph{monochromatic} if 
	$$c\restriction A\times A \setminus \{(a,a)|\; a \in A\}$$ is constant. A set $A\subseteq \omega_1$ \emph{hits} a color $k$, if there are points $a_0\neq a_1 \in A$ for which $c(a_0,a_1)=k$. If $c$ will be symmetric, we will regard it as a function with the domain $[\omega_1]^{\le2}\setminus \{\emptyset\}$. By a \emph{rectangle} we understand a set of the form $I\times J$, for $I,J\subseteq \omega_1$.

	\begin{thm} \label{countableunion}
		Assume $\ma_{\omega_1}$. If $\chi$ is a rectangular model in a language consisting of finitely many symmetric binary relations, then $\chi$ is a countable union of monochromatic sets.
	\end{thm}
	
	\begin{lem} \label{fixcolor}
		Under the hypothesis of Theorem \ref{countableunion}, there exists exactly one color $k \in K$ that is hit on every uncountable subset of $\omega_1$.
	\end{lem}
	
	\begin{proof}
		We proceed by induction on $|K|$. For $|K|=1$ the conclusion is obvious, so let us assume that $|K|>1$, and fix any $k_0 \in K$. First, we show that at least one such color exists. If every uncountable set $X \subseteq \omega_1$ hits $k_0$, we are done. Otherwise, there exists an uncountable set $X$, for which we have
		$$c \restriction [X]^2: \; [X]^2  \rightarrow K \setminus \{k_0\}.$$
		Notice that any uncountable subset of a rectangular model is rectangular itself. Therefore, by the induction hypothesis, there exists some color $k_1 \in K \setminus  \{k_0\}$, that is hit on every uncountable subset of $X$. We will show that $k_1$ is also hit on every uncountable subset of $\omega_1$. Indeed, fix an uncountable set $A \subseteq \omega_1$. By Proposition \ref{propx6}, any bijection $\phi:A \hookrightarrow X$ preserves the coloring $c$ on some uncountable set $A' \subseteq A$. But we know that $\phi[A']$ hits $k_1$, and so does $A'$. By applying Martin's Axiom to the set of finite $k_1$-monochromatic sets, we see that $X$ contains an uncountable $k_1$-monochromatic set (see the proof of Theorem \ref{cliquethm}). This shows that the conclusion cannot hold for any color other than $k_1$.
	\end{proof}
	
	\begin{proof}[Proof of Theorem \ref{countableunion}]
		
		Let us fix a color $k\in K$ that is hit by $c$ on every uncountable subset of $\omega_1$. We will apply Martin's Axiom to the forcing
		$$ \{f: \dom{f} \rightarrow [\omega_1]^{<\omega}|\; \dom{f} \in [\omega]^{<\omega}, \; \forall n \in \dom{f}\; f(n) \text{ is monochromatic of the color $k$}\}.$$
		The proof of the c.c.c. is standard.
	\end{proof}
	
	Recall that a graph $G$ is \emph{countably chromatic} if it can be decomposed into a countable union of anticliques, i.e. there exist a function $h:G\rightarrow \omega$, such that any two connected vertices are mapped into different values. We say that $G$ is \emph{co-countably chromatic} if $G^c$ is countably chromatic.
	
	\begin{cor}
		Assume $\ma_{\omega_1}$. If $G$ is a rectangular graph of size $\omega_1$, then either $G$ or $G^c$ is countably chromatic.
	\end{cor}
	
	The next theorem holds in a broader generality, not only for symmetric relations.
	\begin{thm} \label{countablerecta}
		Assume $\ma_{\omega_1}$. If $\chi=(\omega_1,c)$ is a rectangular model, then $c$ is a countable union of rectangles.
	\end{thm}
	
	\begin{proof}
		We apply Martin's Axiom to the forcing consisting of finite partial functions from $\omega$ to the set $([\omega_1]^{<\omega}\setminus \{\emptyset\})^2$, with the following additional properties
		\begin{itemize}
			\item $\forall n \in \dom{f}$, if $f(n)=(f^n_0,f^n_1)$, then $f^n_0\cap f^n_1=\emptyset$,
			\item $\forall n \in \dom{f}$, $c \restriction f^n_0 \times f^n_1$ is constant.
		\end{itemize}
		
		This forcing produces a partition of $c$ into countable many rectangles: for every pair $(v_0,v_1)$ of distinct ordinals, the set of conditions containing $(v_0,v_1)$ in some coordinate is dense. The non-trivial part is to prove the c.c.c. Let us fix a family of conditions $\{f_\xi|\; \xi<\omega_1\}$. Without loss of generality, they all have the same domain $R \subseteq \omega$. Let us fix $n \in R$. By $\Delta$-system Lemma, we can assume that 
		
		$$\forall \xi<\omega_1 \quad f_\xi(n)=(\{a_1,\ldots,a_s,a_{s+1}^\xi,\ldots,a_t^{\xi}\}, \{b_1,\ldots,b_{s'},b_{s'+1}^\xi,\ldots,b_{t'}^{\xi}\}),$$
		for distinct elements $a_1,\ldots,a_t^{\xi},b_1,\ldots,b_{t'}^{\xi} \in \omega_1$. Next, we can trim our sequence again, so that for all $\xi \neq \eta <\omega_1$, we have
		$$(a_{s+1}^\xi,\ldots,a_t^{\xi},b_{s'+1}^\xi,\ldots,b_{t'}^\xi)\circledast (a_{s+1}^\eta,\ldots,a_t^{\eta},b_{s'+1}^\eta,\ldots,b_{t'}^\eta).$$
		Note that this implies, in particular, that the sets
		$$\{a_{s+1}^\xi,\ldots,a_t^\xi,a_{s+1}^\eta,\ldots,a_t^\eta\},$$
		$$\{b_{s'+1}^\xi,\ldots,b_{t'}^\xi,b_{s'+1}^\eta,\ldots,b_{t'}^\eta\}$$
		are disjoint. We perform the similar trimming procedure for each element of $R$. Finally, we choose any distinct $\xi,\eta < \omega_1$, and define
		$$g(n)=f_{\xi}(n) \cup f_{\eta}(n),$$
		for all $n \in R$.
		It is straightforward that $g \le f_\xi,f_\eta$.
	\end{proof}
	
	Take note that when we say \emph{$G$ is a countable union of rectangles}, we refer to both the connectedness relation, and its complement. This means that the graph is both closed and open in the topology generated by the sides of the (countably many) rectangles. If $|G|\le 2^\omega$, we can easily refine this topology to be Hausdorff. 
	
	\begin{defin}
		A graph $G=(V,E)$ is \emph{open} (resp. \emph{closed}), if $E$ is an open (resp. closed) subset of $V\times V$ with respect to some second countable Hausdorff topology on $V$.\\
		A graph $G=(V,E)$ is \emph{clopen}, if both $E$ and $V\times V \setminus E$ are open subsets of $V\times V$ with respect to some second countable Hausdorff topology on $V$.
	\end{defin}
	
	\begin{thm}
		Let $G=(V,E)$ be any graph. The following conditions are equivalent:
		\begin{enumerate}
			\item $G$ is rectangular,
			\item There exists a c.c.c. partial order forcing that $G$ is a countably or co-countably chromatic clopen graph,
			\item There exists an $\omega_1$-preserving partial order forcing that $G$ is a countably or co-countably chromatic clopen graph.
		\end{enumerate}
	\end{thm}
	\begin{proof} $\text{ }$
		$"1.\Rightarrow 2."$ Using the methods from Section 2, we can define a forcing that forces $\ma_{\omega_1}$ while preserving $1.$. Next, we force $\ma_{|V|}$ the usual way. This second forcing will not destroy rectangularity, because under $\ma_{\omega_1}$, each c.c.c. forcing is Knaster, and thus $G$-c.c. Finally, we apply $\ma_{|V|}$ to the forcing from Theorem \ref{countablerecta} (perhaps replacing $\omega_1$ with $|V|$). \\
		$"2.\Rightarrow 3."$ Clear.\\
		$"3.\Rightarrow 1."$ Let $\mathbb{P}$ denote a forcing from $3.$ By the pigeonhole principle, any countably chromatic clopen graph is a rectangular graph without an uncountable clique. Therefore $G$ will satisfy $1.$ after forcing with $\mathbb{P}$. But $1.$ is a $\Pi_1$ statement in the language of set theory with $\omega_1$ as a parameter, therefore it must have been true in the ground model.
	\end{proof}
	
	It is worth to note a curious analogy. \emph{Gaps} are well-known combinatorial structures, introduced by Hausdorff \cite{hausdorff}, among which much attention was gained by \emph{Suslin gaps} (or destructible gaps/S-gaps). They are characterized by an internal property, that is is equivalent to the property that they remain gaps in any $\omega_1$-preserving generic extension (see for example \cite{stevo} for the details).
	
	\begin{thm}
		Let $G=(V,E)$ be a closed graph. The following conditions are equivalent:
		\begin{enumerate}
			\item For every family $\{(x_0^\xi.\ldots,x_{n-1}^\xi)|\xi <\omega_1\} \subseteq V^n$, where $n<\omega$, there exist $\xi \neq \eta <\omega_1$ such that 
			$$\neg\; x_0^\xi \; E \; x_0^\eta, \;\ldots,\; \neg \; x_{n-1}^\xi \; E \; x_{n-1}^\eta;$$
			\item There exists a c.c.c. partial order forcing that $G$ is countably chromatic;
			\item There exists an $\omega_1$-preserving partial order forcing that $G$ is countably chromatic.
		\end{enumerate}
	\end{thm}
	\begin{proof} $\text{ }$
		$"1.\Rightarrow 2."$ We use the forcing
		$$\mathbb{P}=\{ f:\dom{f}\rightarrow \omega|\; \dom{f} \in [V]^{<\omega} \quad \forall x\neq y \in \dom{f} \; xEy \implies f(x)\neq f(y)\}.$$
		In order to prove the c.c.c. fix a family $\{f_\xi|\; \xi<\omega_1\} \subseteq \mathbb{P}$. Without loss of generality the domains of $f_\xi$ form a $\Delta$-system 
		$$\dom{f_\xi}=(x_1,\ldots,x_k,x_{k+1}^\xi,\ldots,x_m^\xi).$$
		We can moreover assume that the values $f_\xi(x_i)$ and $f_\xi(x_j^\xi)$ are independent of $\xi$, for all $i=1,\ldots,k$, $j=k+1,\ldots,m$. For every $\xi$, we fix a family of pairwise disjoint basic open neighbourhoods 
		$$\{U_x^\xi \ni x|\; x \in \dom{f_\xi}\},$$ with the property that 
		$$\forall x\neq y \in \dom{f_\xi} \quad \neg \; x \; E \; y \implies \forall x' \in U^\xi_x \quad \forall y' \in U^\xi_y\quad \neg  \; x'\; E\; y'. $$
		
		By another trimming, we can assume that the sets $U^\xi_x$ are independent of $\xi$. Finally, we choose $\xi \neq \eta<\omega_1$
		satisfying $1.$ It is standard to check that $f_\xi$ and $f_\eta$ are compatible.\\
		$"2.\Rightarrow 3."$ Clear.\\
		$"3.\Rightarrow 1."$ By the pigeonhole principle, any countably chromatic graph $G$ satisfies $1.$ Therefore $G$ will satisfy $1.$ after forcing with $\mathbb{P}$. But $1.$ is a $\Pi_1$ statement in the language of set theory with $\omega_1$ as a parameter, therefore it must have been true in the ground model.
	\end{proof}

	\begin{thm}
		Assume $\ma_{\omega_1}$. Then any rectangular tournament of size $\omega_1$ is a countable union of transitive tournaments. 
	\end{thm}
	
	\begin{proof}
		We apply Martin's Axiom to the forcing
		$$ \{f: \dom{f} \rightarrow [\omega_1]^{<\omega}|\; \dom{f} \in [\omega]^{<\omega}, \; \forall n \in \dom{f}\; f(n) \text{ is transitive.}\}$$
	\end{proof}
	For the proof of the c.c.c. just recall that transitive tournaments are essentially linear orders, so they satisfy the Strong Amalgamation Property.
	
	\section{Rectangular Structures from $\ch$}
	
	We prove that non-trivial (i.e. HSS) rectangular models exist assuming $\ch$. This is essentially an abstract version of a similar theorem for linear orders, proved in \cite{as}. 
	
	\begin{thm}
		Assume $\ch$. Let $\K$ be a class of structures in a language consisting of countably many binary relations, and assume that $\K$ satisfies RSP and SAP. Then $\K$ has a HSS rectangular model of size $\omega_1$.
	\end{thm}
	\begin{proof}
		Without loss of generality we can assume that every $ A\in\K$ is of the form $(A,c)$, where $A$ is a set, and $c:A^2\rightarrow \omega$. Let $\{M_\alpha|\; \alpha<\omega_1\}$ be a continuous sequence of countable elementary submodels of $H(\omega_1)$, and put $C=\{\alpha<\omega_1|\; M_\alpha\cap \omega_1=\alpha\}$. It is well-known that $C$ is a closed unbounded subset of $\omega_1$. For $\alpha \in C$, we denote by $\alpha'$ the successor of $\alpha$ in $C$. By induction on $\alpha \in C$ we define an increasing sequence of filters $G_\alpha \subseteq \operatorname{Fn}(\alpha,\K,\omega)$, ensuring that for any $\alpha \in C$
		
		$$G_{\alpha'} \subseteq \operatorname{Fn}(\alpha',\K,\omega):G_\alpha\quad \text{ is }\quad \operatorname{Fn}(\alpha',\K,\omega):G_\alpha \text{-generic over } M_{\tau{(\alpha)}},$$
		where $\tau{(\alpha)}$ is big enough, so that $\{\alpha',G_\alpha\} \in M_{\tau({\alpha})}$ (recall that $\operatorname{Fn}(\alpha',\K,\omega):G_\alpha$ denotes the set of those conditions from $\operatorname{Fn}(\alpha',\K,\omega)$ that are compatible with every element from $G_\alpha$). Note that each of these filters is definable from a real, so $\tau(\alpha)$ is well defined. Let $X=(\omega_1,c)$ be the union of all filters $G_\alpha$.\\
		Why is $X$ HSS? We will show that for any $\alpha \in C$, each final segment of $\alpha'$ is a saturating subset of $X$. Let $E\subseteq (\omega_1,c)$ be a finite substructure, with an extension $E\subseteq (E\cup\{e^*\},c^*)$. For $e \in E$ we set
		$$l_e=c^*(e,e^*),$$
		$$l^*_e=c^*(e^*,e).$$
		\paragraph{Case 1.}  Suppose that $E\subseteq \alpha'$. We will show that for every $\beta<\alpha'$ there exists $\delta \in (\beta,\alpha')$, such that for any  $e \in E$, we have 
		$$c(e,\delta)=l_e,$$
		$$c(\delta,e)=l_e^*.$$
		Indeed, it is sufficient to show that the following set (belonging to $M_{\tau(\alpha)}$) is dense in the partial order $\operatorname{Fn}(\alpha',\K,\omega):G_\alpha$:
		$$\{p\in \operatorname{Fn}(\alpha',\K,\omega):G_\alpha|\; \exists \; \delta \in (\beta,\alpha')\quad \forall \; e\in E\quad p \Vdash \dot{c}(e,\delta)=l_i,\; \dot{c}(\delta,e)=l^*_e\}.$$
		To see this, fix some $p \in \operatorname{Fn}(\alpha',\K,\omega):G_\alpha$, and assume that $E \subseteq \dom{p}$ (otherwise we can extend $p$, so that this is satisfied). Next, we fix $\delta \in (\beta,\alpha') \setminus \dom{p}$ and extend $p$ to $p^*$ by amalgamating the diagram
		
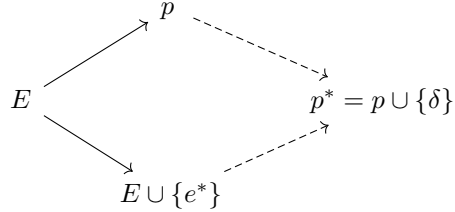
\begin{center}
			\begin{tikzcd}
				& p \ar[dr,dashed] \
				&
				& \\
				E \ar[ur] \ar[dr]
				&
				& p^*=p\cup\{\delta\}
				\\
				& E\cup \{e^*\} \ar[ur,dashed]
				&
				&
			\end{tikzcd}
		\end{center}
		To make sure that $p^*$ is compatible with any condition from $G_\alpha$, notice that for any $q\in G_\alpha$ we can find $p'\le p,q$, not containing $\delta$, and disjointly amalgamate over the diagram
		
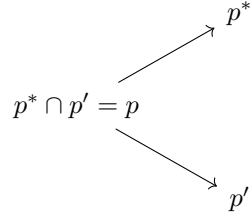
\begin{center}
			\begin{tikzcd}
				& p^*
				&
				& \\
				p^*\cap p'=p \ar[ur] \ar[dr]
				&
				& 
				\\
				& p'
				&
				&
			\end{tikzcd}
		\end{center}
		
		\paragraph{Case 2.} Let $\alpha_1$ be the minimal ordinal in $C$ such that $\alpha_1\ge \alpha$, and $E \subseteq \alpha_1'$. We will prove by induction on $\alpha_1$ that
		$$\forall \; \beta<\alpha' \;\exists \; \delta \in (\beta,\alpha')\; \forall \; e \in E \quad c(e,\delta)=l_e,\; c(\delta,e)=l_e^*.$$
		If $\alpha_1=\alpha$, our task reduces to the previous case, so let as assume that whenever $\alpha_0 \in C \cap \alpha_1$, the conclusion holds for every $E \subseteq \alpha_0'$. In particular, it holds for $E\cap \alpha_1$. Define
		$$I=\{\delta \in (\beta,\alpha')|\; \forall\; e \in E\cap \alpha_1 \quad c(e,\delta)=l_e,\; c(\delta,e)=l_e^*\}.$$
		By the inductive hypothesis, $I$ is nonempty for any value of $\beta$, and so $|I|=\omega$. We must verify that the following set (belonging to $M_{\tau(\alpha_1)}$) is dense in the partial order $\operatorname{Fn}(\alpha_1',\K,\omega):G_{\alpha_1}$:
		$$\{p \in \operatorname{Fn}(\alpha_1',\K,\omega):G_{\alpha_1}|\; \exists \; \delta \in I \quad \forall \; e\in E\setminus \alpha_1 \quad p \Vdash \dot{c}(e,\delta)=l_i,\; \dot{c}(\delta,e)=l^*_e\}.$$
		Notice that this set is in $M_{\tau(\alpha_1)}$. The density argument exactly follows the Case 1, with $\alpha_1$ in place of $\alpha$, and $E\setminus \alpha_1$ in place of $E$.\\
		
		Why is $X$ rectangular? Consider a pairwise disjoint family $$\{(x_0^\xi,\ldots,x_{n-1}^\xi)|\; \xi<\omega_1\} \subseteq X^n.$$ We will be working with a fixed ordinal $\alpha \in C$ with an additional property that $\alpha=N\cap\omega_1$, for some countable elementary submodel $N \prec H(\omega_2)$, which contains the set
		$$\{X,\{(x_0^\xi,\ldots,x_{n-1}^\xi)|\; \xi<\omega_1\}\}.$$
		This property of $N$ ensures the following
		\begin{claim}\text{ }\\
			\begin{enumerate}
				\item $\forall\; \beta<\alpha\quad \{x_0^\beta,\ldots,x_{n-1}^\beta\}\subseteq \alpha$,
				\item $\forall \;\beta>\alpha \quad \{x_0^\beta,\ldots,x_{n-1}^\beta\}\cap \alpha =\emptyset.$
			\end{enumerate}
			\text{ }\\
		\end{claim}
		\begin{proof}[Proof of the Claim]
			\text{ }\\
			\begin{enumerate}
				\item For any $\beta<\alpha$, $\{x_0^\beta,\ldots,x_{n-1}^\beta\} \in N$. Since $\alpha=N\cap \omega_1$, the conclusion follows.
				\item Suppose towards contradiction that $\gamma \in \alpha \cap \{x_0^\beta,\ldots,x_{n-1}^\beta\}$, for some $\beta>\alpha$. By the elementarity,
				$$N\models \exists \; \delta<\omega_1  \quad \gamma \in \{x_0^\delta,\ldots,x_{n-1}^\delta\}.$$ Therefore we can choose $\delta<\alpha$, such that
				$$\gamma \in \{x_0^\beta,\ldots,x_{n-1}^\beta\}\cap \{x_0^\delta,\ldots,x_{n-1}^\delta\}.$$
				But this contradicts the assumption that the tuples are pairwise disjoint.
			\end{enumerate}
		\end{proof}
		
		Fix some ordinal $\beta<\alpha$. We want to show that the following set is dense in $\operatorname{Fn}(\alpha',\K,\omega):G_\alpha$:
		$$\{p \in \operatorname{Fn}(\alpha',\K,\omega):G_\alpha|\; \exists \; l<\omega \quad x^{\alpha+l}_0,\ldots,x^{\alpha+l}_{n-1},x^{\beta}_0,\ldots,x^{\beta}_{n-1} \in \dom{p},\quad$$
		$$p \Vdash (x^{\alpha+l}_0,\ldots,x^{\alpha+l}_{n-1})\circledast (x^{\beta}_0,\ldots,x^{\beta}_{n-1})\}.$$
		
		To see that this is the case, fix $p \in \operatorname{Fn}(\alpha',\K,\omega):G_\alpha$. We can of course assume that
		$\{x^{\beta}_0,\ldots,x^{\beta}_{n-1}\}\subseteq \dom{p}$, and choose $l<\omega$ such that $\{x^{\alpha+l}_0,\ldots,x^{\alpha+l}_{n-1}\}\cap \dom{p}=\emptyset$. Let $p^\circ$ denote the condition obtained by replacing each $x_i^\beta \in \dom{p}$ with $x_i^{\alpha+l}$, and defining relations so that the extensions
		$$p \setminus \{x^{\beta}_0,\ldots,x^{\beta}_{n-1}\} \subseteq p,$$
		and
		$$p \setminus \{x^{\beta}_0,\ldots,x^{\beta}_{n-1}\} \subseteq (p \setminus \{x^{\beta}_0,\ldots,x^{\beta}_{n-1}\}) \cup \{x^{\alpha+l}_0,\ldots,x^{\alpha+l}_{n-1}\}$$
		become isomorphic (via the function $x_i^\beta\mapsto x_i^{\alpha+l}$). We extend $p$ to $p^*$ by applying the RSP to the following diagram of isomorphic extensions
		
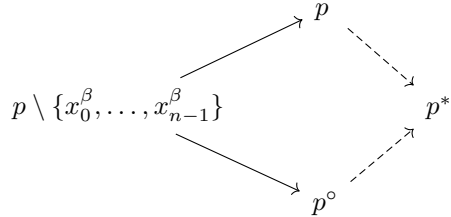
\begin{center}
			\begin{tikzcd}
				& p \ar[dr,dashed] \
				&
				& \\
				p \setminus \{x^{\beta}_0,\ldots,x^{\beta}_{n-1}\} \ar[ur] \ar[dr]
				&
				& p^*
				\\
				& p^\circ \ar[ur,dashed]
				&
				&
			\end{tikzcd}
		\end{center}
		By the RSP, we can amalgamate so that $p^*\Vdash (x^{\alpha+l}_0,\ldots,x^{\alpha+l}_{n-1})\circledast (x^{\beta}_0,\ldots,x^{\beta}_{n-1})$. To see that $p^*$ is compatible with every element of $G_\alpha$, note that whenever $q \in G_\alpha$, we can find $p'\le p,q$ disjoint with $\{x^{\alpha+l}_0,\ldots,x^{\alpha+l}_{n-1}\}$, and then disjointly amalgamate over the diagram
		
		
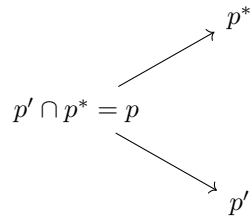
\begin{center}
			\begin{tikzcd}
				& p^*
				&
				& \\
				p'\cap p^*=p \ar[ur] \ar[dr]
				&
				& 
				\\
				& p'
				&
				&
			\end{tikzcd}
		\end{center}

	\end{proof}
	
	\section{Conclusions}
	
	Is the theory we develop entitled to be called "the uncountable Fra\"iss\'e theory with finite supports"? We left it to the reader to decide. Obviously, similarity of forcing notions
	$$\operatorname{Fn}{(\omega,\mathcal{K},\omega)}$$
	and
	$$\operatorname{Fn}{(\omega_1,\mathcal{K},\omega)}$$
	shows that some "genericity" is common to Fra\"iss\'e limits and rectangular models. On the other hand, properties of models added by the latter forcing seem heavily dependent on the set-theoretic background, unlike those of countable homogeneous models. Moreover, it looks like Proposition \ref{propx6} exhibits some strange asymmetry of rectangular models in the presence of $\ma_{\omega_1}$, which distinguishes them from their countable counterparts.

\end{document}